%% file: h2l2_arxiv.tex
\documentclass[twocolumn]{article}

\usepackage{arxiv}

\usepackage[utf8]{inputenc}
\usepackage[T1]{fontenc}
\usepackage{amsmath}
\usepackage{amsfonts}
\usepackage{amssymb}
\usepackage{amsthm}
\usepackage{mathtools}
\usepackage[colorlinks=true, linkcolor=blue]{hyperref}
\usepackage{cleveref}

\newtheorem{thm}{Theorem}[section]

\newtheorem{lem}[thm]{Lemma}

\newtheorem{cor}[thm]{Corollary}

\allowdisplaybreaks%

\numberwithin{equation}{section}

\usepackage{titlesec}
\titlespacing\section{0pt}{12pt plus 3pt minus 3pt}{1pt plus 1pt minus 1pt}
\titlespacing\subsection{0pt}{10pt plus 3pt minus 3pt}{1pt plus 1pt minus 1pt}
\titlespacing\subsubsection{0pt}{8pt plus 3pt minus 3pt}{1pt plus 1pt minus 1pt}

\input{macros.tex}
\let\le\leqslant%
\let\hat\widehat%
\definecolor{lime}{HTML}{A6CE39}
\DeclareRobustCommand{\orcidicon}{
	\begin{tikzpicture}
	\draw[lime, fill=lime] (0,0)
	circle [radius=0.16]
	node[white] {{\fontfamily{qag}\selectfont \tiny ID}};
	\draw[white, fill=white] (-0.0625,0.095)
	circle [radius=0.007];
	\end{tikzpicture}
	\hspace{-2mm}
}
\foreach \x in {A, ..., Z} {%
    \expandafter\xdef\csname orcid\x\endcsname{\noexpand\href{https://orcid.org/\csname orcidauthor\x\endcsname}
      {\noexpand\orcidicon}}
  }

\title{Interpolatory Necessary Optimality Conditions for Reduced-order Modeling
  of Parametric Linear Time-invariant Systems}

\usepackage{authblk}

\author[1]{Petar~Mlinari\'c\orcidA{}}
\author[2]{Peter~Benner\orcidB{}}
\author[3]{Serkan~Gugercin\orcidC{}}

\affil[1]{
  Department of Mathematics,
  Virginia Tech,
  Blacksburg,
  VA 24061
  (\texttt{mlinaric@vt.edu})
}
\affil[2]{
  Max Planck Institute for Dynamics of Complex Technical Systems,
  39106 Magdeburg,
  Germany
  (\texttt{benner@mpi-magdeburg.mpg.de})
}
\affil[3]{
  Department of Mathematics and
  Division of Computational Modeling and Data Analytics,
  Academy of Data Science,
  Virginia Tech,
  Blacksburg,
  VA 24061
  (\texttt{gugercin@vt.edu})
}

\renewcommand{\headeright}{Preprint}
\renewcommand{\undertitle}{Preprint}
\renewcommand{\shorttitle}{Interpolatory Conditions for Parametric Reduced-order
  Modeling}

\hypersetup{
  pdftitle={Interpolatory Necessary Optimality Conditions for Reduced-order
    Modeling of Parametric Linear Time-invariant Systems},
  pdfauthor={
    Petar Mlinari\'c,
    Peter Benner
    Serkan Gugercin
  },
}

\begin{document}

\twocolumn[ 
  \begin{@twocolumnfalse} 

\maketitle

\begin{abstract}
  \input{h2l2_abstract.tex}
\end{abstract}


\vspace{0.35cm}

  \end{@twocolumnfalse} 
] 

\input{h2l2_content.tex}

\bibliographystyle{alphaurl}
\bibliography{my}
\end{document}

%% file: macros.tex
\usepackage{amsmath}
\usepackage{amsfonts}
\usepackage{amssymb}
\usepackage{mathtools}
\usepackage{tikz}
\usepackage{pgfplots}
\usepackage{dsfont}

\usepackage{acronym}

\acrodef{fom}[FOM]{full-order model}
\acrodef{rom}[ROM]{reduced-order model}
\acrodef{strom}[StROM]{structured \ac{rom}}
\acrodef{diagstrom}[D-StROM]{diagonal \ac{strom}}

\acrodef{siso}[SISO]{single-input single-output}
\acrodef{mimo}[MIMO]{multiple-input multiple-output}

\acrodef{lti}[LTI]{linear time-invariant}
\acrodefindefinite{lti}{an}{a}

\acrodef{ph}[pH]{port-Hamiltonian}

\acrodef{irka}[IRKA]{iterative rational Krylov algorithm}

\newcommand{\nfom}{\ensuremath{n}}
\newcommand{\nin}{\ensuremath{n_{\textnormal{i}}}}
\newcommand{\nout}{\ensuremath{n_{\textnormal{o}}}}

\newcommand{\Af}{\ensuremath{\mathcal{A}}}
\newcommand{\Bf}{\ensuremath{\mathcal{B}}}
\newcommand{\Cf}{\ensuremath{\mathcal{C}}}
\newcommand{\Ef}{\ensuremath{\mathcal{E}}}

\newcommand{\cAf}{\ensuremath{A}}
\newcommand{\cBf}{\ensuremath{B}}
\newcommand{\cCf}{\ensuremath{C}}

\newcommand{\yf}{\ensuremath{H}}

\newcommand{\nrom}{\ensuremath{r}}

\newcommand{\Ar}{\ensuremath{\hat{\mathcal{A}}}}
\newcommand{\Br}{\ensuremath{\hat{\mathcal{B}}}}
\newcommand{\Cr}{\ensuremath{\hat{\mathcal{C}}}}
\newcommand{\Er}{\ensuremath{\hat{\mathcal{E}}}}
\newcommand{\Fr}{\ensuremath{\hat{\mathcal{F}}}}
\newcommand{\Gr}{\ensuremath{\hat{\mathcal{G}}}}
\newcommand{\Kr}{\ensuremath{\hat{\mathcal{K}}}}

\newcommand{\cAr}{\ensuremath{\hat{A}}}
\newcommand{\cBr}{\ensuremath{\hat{B}}}
\newcommand{\cCr}{\ensuremath{\hat{C}}}
\newcommand{\cEr}{\ensuremath{\hat{E}}}
\newcommand{\cFr}{\ensuremath{\hat{F}}}
\newcommand{\cGr}{\ensuremath{\hat{G}}}
\newcommand{\cKr}{\ensuremath{\hat{K}}}

\newcommand{\car}{\ensuremath{\hat{\alpha}}}
\newcommand{\cbr}{\ensuremath{\hat{\beta}}}
\newcommand{\ccr}{\ensuremath{\hat{\gamma}}}
\newcommand{\cer}{\ensuremath{\hat{\varepsilon}}}
\newcommand{\cfr}{\ensuremath{\hat{\zeta}}}
\newcommand{\cgr}{\ensuremath{\hat{\eta}}}
\newcommand{\ckr}{\ensuremath{\hat{\kappa}}}

\newcommand{\nAr}{\ensuremath{n_{\Ar}}}
\newcommand{\nBr}{\ensuremath{n_{\Br}}}
\newcommand{\nCr}{\ensuremath{n_{\Cr}}}
\newcommand{\nEr}{\ensuremath{n_{\Er}}}
\newcommand{\nFr}{\ensuremath{n_{\Fr}}}
\newcommand{\nGr}{\ensuremath{n_{\Gr}}}
\newcommand{\nKr}{\ensuremath{n_{\Kr}}}

\newcommand{\xr}{\ensuremath{\hat{X}}}
\newcommand{\yr}{\ensuremath{\hat{H}}}

\newcommand{\hA}{\ensuremath{\hat{A}}}

\newcommand{\hG}{\ensuremath{\hat{G}}}
\newcommand{\hH}{\ensuremath{\hat{H}}}

\newcommand{\hx}{\ensuremath{\hat{x}}}
\newcommand{\hy}{\ensuremath{\hat{y}}}

\newcommand{\hcH}{\ensuremath{\hat{\cH}}}

\newcommand{\cH}{\ensuremath{\mathcal{H}}}
\newcommand{\cL}{\ensuremath{\mathcal{L}}}
\newcommand{\Htwo}{\ensuremath{\cH_{2}}}
\newcommand{\htwo}{\ensuremath{h_{2}}}
\newcommand{\Ltwo}{\ensuremath{\cL_{2}}}
\newcommand{\Hinf}{\ensuremath{\cH_{\infty}}}
\newcommand{\Linf}{\ensuremath{\cL_{\infty}}}

\newcommand{\HtwoLtwo}{\ensuremath{\cH_{2} \otimes \cL_{2}}}

\newcommand{\pp}{\ensuremath{\mathsf{p}}}
\newcommand{\pset}{\ensuremath{\mathcal{P}}}
\newcommand{\npar}{\ensuremath{n_{\pp}}}
\newcommand{\ppp}{\ensuremath{\mathsf{q}}}
\newcommand{\psetp}{\ensuremath{\mathcal{Q}}}
\newcommand{\nparp}{\ensuremath{n_{\ppp}}}
\newcommand{\CCparp}{\ensuremath{\CC^{\nparp}}}
\newcommand{\RRparp}{\ensuremath{\RR^{\nparp}}}

\newcommand{\CC}{\ensuremath{\mathbb{C}}}

\newcommand{\RR}{\ensuremath{\mathbb{R}}}

\newcommand{\CCpar}{\ensuremath{\CC^{\npar}}}

\newcommand{\CCoi}{\ensuremath{\CC^{\nout \times \nin}}}

\newcommand{\CCor}{\ensuremath{\CC^{\nout \times \nrom}}}
\newcommand{\CCri}{\ensuremath{\CC^{\nrom \times \nin}}}

\newcommand{\CCrr}{\ensuremath{\CC^{\nrom \times \nrom}}}

\newcommand{\RRf}{\ensuremath{\RR^{\nfom}}}
\newcommand{\RRi}{\ensuremath{\RR^{\nin}}}
\newcommand{\RRo}{\ensuremath{\RR^{\nout}}}
\newcommand{\RRr}{\ensuremath{\RR^{\nrom}}}
\newcommand{\RRff}{\ensuremath{\RR^{\nfom \times \nfom}}}
\newcommand{\RRfi}{\ensuremath{\RR^{\nfom \times \nin}}}
\newcommand{\RRof}{\ensuremath{\RR^{\nout \times \nfom}}}

\newcommand{\RRor}{\ensuremath{\RR^{\nout \times \nrom}}}
\newcommand{\RRri}{\ensuremath{\RR^{\nrom \times \nin}}}

\newcommand{\RRrr}{\ensuremath{\RR^{\nrom \times \nrom}}}

\DeclarePairedDelimiter{\myparen}{\lparen}{\rparen}

\DeclarePairedDelimiter{\abs}{\lvert}{\rvert}

\DeclarePairedDelimiterXPP{\normtwo}[1]{}{\lVert}{\rVert}{_{2}}{#1}
\DeclarePairedDelimiterXPP{\normF}[1]{}{\lVert}{\rVert}{_{\operatorname{F}}}{#1}
\DeclarePairedDelimiterXPP{\normHinf}[1]{}{\lVert}{\rVert}{_{\Hinf}}{#1}
\DeclarePairedDelimiterXPP{\normHtwoLtwo}[1]{}{\lVert}{\rVert}{_{\HtwoLtwo}}{#1}
\DeclarePairedDelimiterXPP{\normHtwo}[1]{}{\lVert}{\rVert}{_{\Htwo}}{#1}
\DeclarePairedDelimiterXPP{\normLinfLtwo}[1]{}{\lVert}{\rVert}{_{\Linf \otimes \Ltwo}}{#1}
\DeclarePairedDelimiterXPP{\normLinfmu}[1]{}{\lVert}{\rVert}{_{\Linf(\pset, \measure)}}{#1}
\DeclarePairedDelimiterXPP{\normLinf}[1]{}{\lVert}{\rVert}{_{\Linf}}{#1}
\DeclarePairedDelimiterXPP{\normLtwomu}[1]{}{\lVert}{\rVert}{_{\Ltwo(\pset, \measure)}}{#1}
\DeclarePairedDelimiterXPP{\normLtwo}[1]{}{\lVert}{\rVert}{_{\Ltwo}}{#1}
\DeclarePairedDelimiterXPP{\normhtwo}[1]{}{\lVert}{\rVert}{_{\htwo}}{#1}
\DeclarePairedDelimiterXPP{\ip}[2]{}{\langle}{\rangle}{}{#1, #2}
\DeclarePairedDelimiterXPP{\ipF}[2]{}{\langle}{\rangle}{_{\operatorname{F}}}{#1, #2}
\DeclarePairedDelimiterXPP{\ipHtwo}[2]{}{\langle}{\rangle}{_{\Htwo}}{#1, #2}
\DeclarePairedDelimiterXPP{\iphtwo}[2]{}{\langle}{\rangle}{_{\htwo}}{#1, #2}
\DeclarePairedDelimiterXPP{\ipLtwo}[2]{}{\langle}{\rangle}{_{\Ltwo}}{#1, #2}
\DeclarePairedDelimiterXPP{\ipLtwomu}[2]{}{\langle}{\rangle}{_{\Ltwo(\pset, \measure)}}{#1, #2}
\DeclarePairedDelimiterXPP{\mydiag}[1]{\operatorname{diag}}{\lparen}{\rparen}{}{#1}
\DeclarePairedDelimiterXPP{\trace}[1]{\operatorname{tr}}{\lparen}{\rparen}{}{#1}
\DeclarePairedDelimiterXPP{\vecop}[1]{\operatorname{vec}}{\lparen}{\rparen}{}{#1}
\DeclarePairedDelimiterXPP{\Real}[1]{\operatorname{Re}}{\lparen}{\rparen}{}{#1}
\DeclarePairedDelimiterXPP{\dist}[1]{\operatorname{dist}}{\lparen}{\rparen}{}{#1}
\DeclarePairedDelimiterXPP{\len}[1]{\operatorname{length}}{\lparen}{\rparen}{}{#1}
\DeclarePairedDelimiterXPP{\myspan}[1]{\operatorname{span}}{\lbrace}{\rbrace}{}{#1}

\newcommand{\measure}{\ensuremath{\mu}}
\newcommand{\measurep}{\ensuremath{\nu}}

\DeclareMathOperator{\dif}{d\!}

\DeclareMathOperator*{\esssup}{ess\,sup}

\newcommand{\difm}[1]{\ensuremath{\dif{\measure(#1)}}}
\newcommand{\difmp}[1]{\ensuremath{\dif{\measurep(#1)}}}

\newcommand{\fundef}[3]{\ensuremath{#1 \colon #2 \to #3}}

\newcommand{\tran}{^{\operatorname{T}}}

\newcommand{\herm}{^{*}}

\newcommand{\imag}{\boldsymbol{\imath}}

\definecolor{mplC0}{HTML}{1F77B4}
\definecolor{mplC1}{HTML}{FF7F0E}
\definecolor{mplC2}{HTML}{2CA02C}
\definecolor{mplC3}{HTML}{D62728}
\definecolor{mplC4}{HTML}{9467BD}
\definecolor{mplC5}{HTML}{8C564B}
\definecolor{mplC6}{HTML}{E377C2}
\definecolor{mplC7}{HTML}{7F7F7F}
\definecolor{mplC8}{HTML}{BCBD22}
\definecolor{mplC9}{HTML}{17BECF}

\pgfplotsset{compat=1.10}
\pgfplotscreateplotcyclelist{mpl}{
  mplC0, very thick \\
  mplC1, very thick, densely dashed \\
  mplC2, very thick, densely dotted \\
  mplC3, very thick, densely dashdotted \\
}
\pgfplotscreateplotcyclelist{mplmark}{
  mplC0, very thick, mark=o, every mark/.append style={solid} \\
  mplC1, very thick, densely dashed, mark=x, every mark/.append style={solid}, mark size=3 \\
  mplC2, very thick, densely dotted, mark=square, every mark/.append style={solid} \\
  mplC3, very thick, densely dashdotted, mark=diamond, every mark/.append style={solid} \\
}
\pgfplotscreateplotcyclelist{mplmarkonly}{
  draw=mplC0, fill=mplC0, only marks, mark=*, mark size=1 \\
  draw=mplC1, fill=mplC1, only marks, mark=square*, mark size=1 \\
  draw=mplC2, fill=mplC2, only marks, mark=triangle*, mark size=1.5 \\
}

%% file: h2l2_abstract.tex
Interpolatory necessary optimality conditions for
$\Htwo$-optimal reduced-order modeling of
non-parametric linear time-invariant (LTI) systems
are known and well-investigated.
In this work, using the general framework of
$\Ltwo$-optimal reduced-order modeling of
parametric stationary problems,
we derive interpolatory $\HtwoLtwo$-optimality conditions for
parametric LTI systems with a general pole-residue form.
We then specialize this result to recover known conditions for
systems with parameter-independent poles and
develop new conditions for a certain class of systems
with parameter-dependent poles.

%% file: h2l2_keywords.tex
reduced-order modeling,
parametric systems,
optimization,
linear systems

%% file: h2l2_content.tex
\section{Introduction}

Consider a parametric \ac{lti} system (\ac{fom})
\begin{subequations}\label{eq:p-fom}
  \begin{align}
    \Ef(\ppp) \dot{x}(t, \ppp)
     & = \Af(\ppp) x(t, \ppp) + \Bf(\ppp) u(t), \\*
    y(t, \ppp)
     & = \Cf(\ppp) x(t, \ppp),
  \end{align}
\end{subequations}
where
$\ppp \in \psetp \subseteq \RRparp$ is the parameter vector;
$u(t) \in \RRi$ is the input;
$x(t, \ppp) \in \RRf$ is the state;
$y(t, \ppp) \in \RRo$ is the output; and
$\Ef(\ppp), \Af(\ppp) \in \RRff$,
$\Bf(\ppp) \in \RRfi$, and
$\Cf(\ppp) \in \RRof$ are parametric matrices.
Given the \ac{fom} in~\eqref{eq:p-fom},
the goal of parametric reduced-order modeling is
to find a reduced parametric \ac{lti} system (\ac{rom})
\begin{subequations}\label{eq:p-rom}
  \begin{align}
    \Er(\ppp) \dot{\hx}(t, \ppp)
     & = \Ar(\ppp) \hx(t, \ppp) + \Br(\ppp) u(t), \\*
    \hy(t, \ppp)
     & = \Cr(\ppp) \hx(t, \ppp),
  \end{align}
\end{subequations}
where
$\hx(t, \ppp) \in \RRr$ is the reduced state with $\nrom \ll \nfom$;
$\hy(t, \ppp) \in \RRo$ is the approximate output; and
$\Er(\ppp), \Ar(\ppp) \in \RRrr$,
$\Br(\ppp) \in \RRri$, and
$\Cr(\ppp) \in \RRor$ are the reduced parametric matrices,
such that $\hy$ approximates $y$
for a wide range of inputs $u$ and
a set of parameters $\ppp \in \psetp$.
Parametric dynamical systems are ubiquitous in applications ranging from inverse
problems to uncertainty quantification to optimization and model reduction of
parametric systems has been a major research topic;
we refer the reader to, e.g.,~\cite{BenGW15,BenGQ+20b} for more details.

Both the \ac{fom} and \ac{rom} can be fully described by their (parametric)
transfer functions, given by, respectively,
\begin{align}
  \nonumber
  H(s, \ppp)
   & =
  \Cf(\ppp) \myparen*{s \Ef(\ppp) - \Af(\ppp)}^{-1} \Bf(\ppp)
  \ \text{ and} \\
  \label{eq:p-rom-tf}
  \hH(s, \ppp)
   & =
  \Cr(\ppp) \myparen*{s \Er(\ppp) - \Ar(\ppp)}^{-1} \Br(\ppp).
\end{align}
As in any approximation problem,
one needs a metric to judge the quality of the approximation.
For \emph{non-parametric} \ac{lti} systems, i.e.,
when $\Ef, \Af, \Bf, \Cf$ are constant matrices,
the $\Htwo$-norm has been one of the most commonly used metrics in (optimal)
reduced-order modeling~\cite{AntBG20,GugAB08,MeiL67,Wil70}.
For parametric \ac{lti} systems we consider here,
the $\HtwoLtwo$-norm introduced in~\cite{BauBBG11} provides a natural extension.
The goal of $\HtwoLtwo$-optimal reduced-order modeling
is to find \iac{rom} that (locally) minimizes the $\HtwoLtwo$ error
\begin{equation}\label{eq:h2l2-norm}
  \normHtwoLtwo*{H - \hH}
  =
  \myparen*{
    \int_{\psetp}
    \normHtwo*{H(\cdot, \ppp) - \hH(\cdot, \ppp)}^2
    \difmp{\ppp}
  }^{1/2},
\end{equation}
where $\measurep$ is a measure over $\psetp$ and the $\Htwo$ norm is given as
\begin{equation*}
  \normHtwo*{H(\cdot, \ppp)}
  =
  \myparen*{
    \frac{1}{2 \pi}
    \int_{-\infty}^{\infty}
    \normF*{H(\imag \omega, \ppp)}^2
    \dif{\omega}
  }^{1/2}.
\end{equation*}
The $\HtwoLtwo$ error gives an upper bound for the output error
\begin{equation*}
  \normLinfLtwo*{y - \hy}
  \le
  \normHtwoLtwo*{H - \hH}
  \normLtwo{u},
\end{equation*}
where
\begin{equation*}
  \normLinfLtwo*{y}
  =
  \myparen*{
    \int_{\psetp}
    \normLinf{y(\cdot, \ppp)}^2
    \difmp{\ppp}
  }^{1/2}
\end{equation*}
is the $\Linf \otimes \Ltwo$ norm of the output,
further justifying the use of the $\HtwoLtwo$ norm in parametric reduced-order
modeling.

The $\HtwoLtwo$ norm with $\measurep$ as the Lebesgue measure was introduced in
Baur~et~al.~\cite{BauBBG11}.
There, for the special case where $\Er$ and $\Ar$ are
\emph{parameter-independent},
the $\HtwoLtwo$-optimal reduced-order modeling problem was converted to a
non-parametric $\Htwo$-optimal reduced-order modeling problem and
interpolatory optimality conditions could be established.
For another simplified problem where the poles of $\hH$ do \emph{not} vary with
the parameter $\ppp \in \{|\ppp| = 1 \} \subset \mathbb{C}$,
Grimm~\cite{Gri18} used an $\HtwoLtwo$ norm and
derived interpolatory conditions and proposed an optimization algorithm.

A common assumption in parametric reduced-order modeling methods is
parameter-separability.
For the \ac{rom}~\eqref{eq:p-rom},
this would mean that the reduced quantities can be written as
\begin{subequations}\label{eq:lti-p-sep-form}
  \begin{alignat}{4}
    \label{eq:lti-p-sep-form-a}
    \Er(\ppp)
     & = \sum_{\ell = 1}^{\nEr} \cer_{\ell}(\ppp) \cEr_{\ell},
     & \quad
    \Ar(\ppp)
     & = \sum_{i = 1}^{\nAr} \car_i(\ppp) \cAr_i,              \\
    \label{eq:lti-p-sep-form-b}
    \Br(\ppp)
     & = \sum_{j = 1}^{\nBr} \cbr_j(\ppp) \cBr_j,
     & \quad
    \Cr(\ppp)
     & = \sum_{k = 1}^{\nCr} \ccr_k(\ppp) \cCr_k,
  \end{alignat}
\end{subequations}
for some
functions $\fundef{\cer_{\ell}, \car_i, \cbr_j, \ccr_k}{\psetp}{\RR}$,
constant matrices $\cEr_{\ell}$, $\cAr_i$, $\cBr_j$, $\cCr_k$, and
positive integers $\nEr, \nAr, \nBr, \nCr$.
We call \iac{rom} of this form \iac{strom}.
This form has also been considered in $\HtwoLtwo$-optimal reduced-order modeling
methods.
In particular,
Petersson~\cite{Pet13} considered the case of a discretized $\HtwoLtwo$ norm,
i.e., where $\measurep$ is a sum of Dirac measures,
proposing an optimization algorithm to find
a locally $\HtwoLtwo$-optimal \ac{rom}.
Additionally,
Hund~et~al.~\cite{HunMMS22} proposed an optimization algorithm for
$\HtwoLtwo$-optimal reduced-order modeling using quadrature for the case of
Lebesgue measure.
Both of these works used matrix equation-based, Wilson-type
conditions~\cite{Wil70} and not interpolation.

The $\HtwoLtwo$-norm was also used by Brunsch~\cite{Bru17} to derive error
bounds within a reduced-order modeling framework for parametric \ac{lti} systems
with symmetric positive definite $\Ef(\ppp)$ and $-\Af(\ppp)$.
The method is based on sparse-grid interpolation in the parameter domain.
It satisfies (Hermite) interpolation conditions and preserves stability,
but has no proven optimality properties.
See also~\cite{IonA14,BraGZ+22,RodBG23} for some data-driven approaches.

In our recent work on $\Ltwo$-optimal reduced-order modeling~\cite{MliG23a},
we covered both \ac{lti} systems and parametric stationary problems.
We developed interpolatory necessary optimality
conditions in~\cite{MliG23b} for certain types of \acp{strom},
including \emph{non-parametric} \ac{lti} systems and
parametric stationary problems.
We also showed that the interpolatory conditions of~\cite{Gri18}
can be derived from our generalized $\Ltwo$-optimality conditions.
However, as stated before,~\cite{Gri18} assumes the poles are fixed.

Therefore, unlike for the non-parametric \ac{lti} problems
for which interpolatory optimality conditions for $\Htwo$ model reduction have
been well-established~\cite{MeiL67,GugAB08,AntBG20},
there is a significant gap in the development of interpolatory optimality
conditions for $\HtwoLtwo$-optimal parametric \ac{rom} construction,
except for the special cases mentioned above.
Our goal in this paper is to close this gap and
to develop interpolatory optimality conditions
for the more general setting of parametric \ac{lti} systems.
Additionally, we show that our analysis contains the earlier conditions
from~\cite{BauBBG11} as a special case.

We provide background in~\Cref{sec:background}.
While \Cref{sec:h2l2-diag} covers the general parametric \acp{diagstrom} case,
\Cref{sec:param-io,sec:param-dyn} focus on simplified cases,
leading to optimality conditions that can be directly linked to the bitangential
Hermite interpolation framework.
We conclude with~\Cref{sec:conclusion}.

\section{Background}%
\label{sec:background}

Here we recall one of the main results of~\cite{MliBG23},
specifically, the necessary $\Ltwo$-optimality conditions for \acp{diagstrom}
(based on prior work in~\cite{MliG23a,MliG23b}),
which will form the foundation of our analysis.

Given a parameter-to-output mapping
\begin{equation*}
  \fundef{\yf}{\pset}{\CCoi},
\end{equation*}
the goal in~\cite{MliBG23} is to construct \iac{strom}
\begin{subequations}\label{eq:rom}
  \begin{align}
    \Kr(\pp) \xr(\pp) & = \Fr(\pp),          \\*
    \yr(\pp)          & = \Gr(\pp) \xr(\pp),
  \end{align}
\end{subequations}
with a parameter-separable form
\begin{equation}\label{eq:rom-p-sep-form}
  \Kr(\pp) = \sum_{i = 1}^{\nKr} \ckr_i(\pp) \cKr_i, \ \
  \Fr(\pp) = \sum_{j = 1}^{\nFr} \cfr_j(\pp) \cFr_j, \ \
  \Gr(\pp) = \sum_{k = 1}^{\nGr} \cgr_k(\pp) \cGr_k,
\end{equation}
where
$\xr(\pp) \in \CCri$ is the reduced state,
$\yr(\pp) \in \CCoi$ is the approximate output,
$\Kr(\pp) \in \CCrr$,
$\Fr(\pp) \in \CCri$,
$\Gr(\pp) \in \CCor$,
$\fundef{\ckr_i, \cfr_j, \cgr_k}{\pset}{\CC}$,
$\cKr_i \in \CCrr$,
$\cFr_j \in \CCri$, and
$\cGr_k \in \CCor$.
The goal is to construct $\Kr(\pp)$, $\Fr(\pp)$, and $\Gr(\pp)$ such that
$\yr(\pp) = \Gr(\pp) \Kr(\pp)^{-1} \Fr(\pp)$ is an optimal $\Ltwo$-approximation
to the original mapping $\yf(\pp)$, i.e.,
\begin{equation}\label{eq:l2-norm}
  \normLtwo*{\yf - \yr}
  =
  \myparen*{
    \int_{\pset}
    \normF*{\yf(\pp) - \yr(\pp)}^2
    \difm{\pp}
  }^{1/2}
\end{equation}
is minimized
where $\measure$ is a measure over $\pset$.
The $\HtwoLtwo$ norm~\eqref{eq:h2l2-norm} is a special case of the
$\Ltwo$-norm~\eqref{eq:l2-norm} for appropriately defined $\pp$ and $\measure$,
a fact we exploit in \Cref{sec:h2l2-diag,sec:param-io,sec:param-dyn}.
We will use the notation $(\cKr_i, \cFr_j, \cGr_k)$ to denote the \ac{strom}
specified by~\eqref{eq:rom} and~\eqref{eq:rom-p-sep-form}.

We assume $(\cKr_i, \cFr_j, \cGr_k)$ is \iac{diagstrom}, i.e.,
all $\cKr_i$'s are diagonal,
and in return so is $\Kr(\pp)$ in~\eqref{eq:rom}.
Then $\yr$ has a ``pole-residue'' form
\begin{equation}\label{eq:diag-pole-res}
  \yr(\pp)
  = \Gr(\pp) {\Kr(\pp)}^{-1} \Fr(\pp)
  = \sum_{\ell = 1}^{\nrom}
  \frac{g_{\ell}(\pp) f_{\ell}(\pp)\herm}{k_{\ell}(\pp)},
\end{equation}
where
$k_{\ell}(\pp)$ is the $\ell$th diagonal entry of $\Kr(\pp)$,
$f_{\ell}(\pp) = \Fr(\pp)\herm e_{\ell}$, and
$g_{\ell}(\pp) = \Gr(\pp) e_{\ell}$.
With this pole-residue form in hand,
we have the optimality conditions for \acp{diagstrom}
(Corollary~2.4 in~\cite{MliBG23}).
\begin{thm}\label{thm:l2-cond-diag}
  Suppose that
  $\pset \subseteq \CCpar$;
  $\measure$ is a measure over $\pset$;
  the function $\yf$ is in $\Ltwo(\pset, \measure; \CCoi)$;
  functions $\fundef{\ckr_i, \cfr_j, \cgr_k}{\pset}{\CC}$ are measurable and
  satisfy
  \begin{equation}\label{eq:abc-l2-bounded}
    \int_{\pset}
    \myparen*{
      \frac{
        \sum_{j = 1}^{\nFr} \abs*{\cfr_j(\pp)}
        \sum_{k = 1}^{\nGr} \abs*{\cgr_k(\pp)}
      }{
        \sum_{i = 1}^{\nKr} \abs*{\ckr_i(\pp)}
      }
    }^{2}
    \difm{\pp}
    < \infty,
  \end{equation}
  $\cKr_i \in \CCrr$,
  $\cFr_j \in \CCri$,
  $\cGr_k \in \CCor$; and
  \begin{equation}\label{eq:ainv-bounded}
    \esssup_{\pp \in \pset} \, \normF*{\ckr_i(\pp) {\Kr(\pp)}^{-1}}
    < \infty, \quad
    i = 1, 2, \ldots, \nKr,
  \end{equation}
  where $\Kr$ is as in~\eqref{eq:rom-p-sep-form}.
  Furthermore, let $(\cKr_i, \cFr_j, \cGr_k)$ be an $\Ltwo$-optimal
  \ac{diagstrom} of $\yf$ with $\yr$ as in~\eqref{eq:diag-pole-res}.
  Then
  \begin{subequations}\label{eq:l2-cond-diag}
    \begin{gather}
      \label{eq:l2-cond-diag-C}
      \int_{\pset}
      \frac{
        \overline{\cgr_k(\pp)}
        \yf(\pp)
        f_{\ell}(\pp)
      }{
        \overline{k_{\ell}(\pp)}
      }
      \difm{\pp}
      =
      \int_{\pset}
      \frac{
        \overline{\cgr_k(\pp)}
        \yr(\pp)
        f_{\ell}(\pp)
      }{
        \overline{k_{\ell}(\pp)}
      }
      \difm{\pp}, \\
      \label{eq:l2-cond-diag-B}
      \int_{\pset}
      \frac{
        \overline{\cfr_j(\pp)}
        g_{\ell}(\pp)\herm
        \yf(\pp)
      }{
        \overline{k_{\ell}(\pp)}
      }
      \difm{\pp}
      =
      \int_{\pset}
      \frac{
        \overline{\cfr_j(\pp)}
        g_{\ell}(\pp)\herm
        \yr(\pp)
      }{
        \overline{k_{\ell}(\pp)}
      }
      \difm{\pp}, \\
      \label{eq:l2-cond-diag-A}
      \begin{aligned}
         &
        \int_{\pset}
        \frac{
          \overline{\ckr_i(\pp)}
          g_{\ell}(\pp)\herm
          \yf(\pp)
          f_{\ell}(\pp)
        }{
          \overline{k_{\ell}(\pp)}^2
        }
        \difm{\pp} \\*
         & \quad =
        \int_{\pset}
        \frac{
          \overline{\ckr_i(\pp)}
          g_{\ell}(\pp)\herm
          \yr(\pp)
          f_{\ell}(\pp)
        }{
          \overline{k_{\ell}(\pp)}^2
        }
        \difm{\pp},
      \end{aligned}
    \end{gather}
  \end{subequations}
  for
  $i = 1, 2, \ldots, \nKr$,
  $j = 1, 2, \ldots, \nFr$,
  $k = 1, 2, \ldots, \nGr$, and
  $\ell = 1, 2, \ldots, \nrom$.
\end{thm}
\Cref{thm:l2-cond-diag} establishes the interpolatory optimality
conditions~\eqref{eq:l2-cond-diag} for $\Ltwo$-optimal approximation.
We showed in~\cite{MliBG23} that various structured reduced-order modeling
problems appear as a special case of \Cref{thm:l2-cond-diag} and
derived interpolatory optimality conditions for important classes of
\emph{non-parametric} structured \ac{lti} systems.
In this paper, we extend this analysis to parametric \ac{lti} systems.

\section{\texorpdfstring{$\HtwoLtwo$}{H2xL2}-optimal Parametric Interpolation}%
\label{sec:h2l2-diag}

Here we use Theorem~\ref{thm:l2-cond-diag} to derive interpolatory conditions for
$\HtwoLtwo$-optimal reduced-order approximation of the \ac{fom}~\eqref{eq:p-fom}
using \acp{diagstrom}.
But first we need to establish what the assumptions~\eqref{eq:abc-l2-bounded}
and~\eqref{eq:ainv-bounded} appearing in $\Ltwo$-optimal approximation for
the structure~\eqref{eq:rom-p-sep-form} correspond to in the case of
$\HtwoLtwo$ approximation with the structures of \acp{strom}
in~\eqref{eq:lti-p-sep-form}.
\begin{lem}\label{lem:assumptions}
  Let $\pp = (s, \ppp)$, $\pset = \imag \RR \times \psetp$, and
  $\measure = \frac{1}{2 \pi} \lambda_{\imag \RR} \times \measurep$
  where $\lambda_{\imag \RR}$ is the Lebesgue measure over $\imag \RR$ and
  $\measurep$ is a measure over $\psetp \subseteq \CCparp$.
  Furthermore, for \ac{strom} in~\eqref{eq:lti-p-sep-form},
  let the functions
  $\fundef{\cer_{\ell}, \car_i, \cbr_j, \ccr_k}{\psetp}{\CC}$ be measurable.
  Then the condition
  \begin{equation}\label{eq:eabc-l2-bounded}
    \int_{\psetp}
    \frac{
      \myparen*{
        \sum_{j = 1}^{\nBr} \abs*{\cbr_j(\ppp)}
        \sum_{k = 1}^{\nCr} \abs*{\ccr_k(\ppp)}
      }^{2}
    }{
      \sum_{\ell = 1}^{\nEr} \abs*{\cer_{\ell}(\ppp)}
      \sum_{i = 1}^{\nAr} \abs*{\car_i(\ppp)}
    }
    \difmp{\ppp}
    < \infty \\
  \end{equation}
  is equivalent to~\eqref{eq:abc-l2-bounded}, and the conditions
  \begin{subequations}\label{eq:linf-bounded}
    \begin{align}
      \label{eq:e-linf-bounded}
      \esssup_{\ppp \in \psetp} \,
      \abs*{\cer_{\ell}(\ppp)}
      \normLinf*{s \myparen*{s \Er(\ppp) - \Ar(\ppp)}^{-1}}
       & < \infty, \\
      \label{eq:a-linf-bounded}
      \esssup_{\ppp \in \psetp} \,
      \abs*{\car_i(\ppp)}
      \normLinf*{\myparen*{s \Er(\ppp) - \Ar(\ppp)}^{-1}}
       & < \infty,
    \end{align}
  \end{subequations}
  for $\ell = 1, \ldots, \nEr$ and $i = 1, \ldots, \nAr$,
  are equivalent to~\eqref{eq:ainv-bounded}.
\end{lem}
\begin{proof}
  First note that with the choices of
  $\pp = (s, \ppp)$,
  $\pset = \imag \RR \times \psetp$, and
  $\measure = \frac{1}{2 \pi} \lambda_{\imag \RR} \times \measurep$,
  the $\Ltwo$-norm in~\eqref{eq:l2-norm}
  recovers the $\HtwoLtwo$norm in~\eqref{eq:h2l2-norm}.
  Now note that the integral in~\eqref{eq:abc-l2-bounded},
  for the \ac{strom} $(s \Er(\ppp) - \Ar(\ppp), \Br(\ppp), \Cr(\ppp))$
  as in~\eqref{eq:lti-p-sep-form},
  takes the form
  \begin{equation*}
    \int_{\psetp}
    \int_{-\infty}^{\infty}
    \myparen*{
      \frac{
        \sum_{j = 1}^{\nBr} \abs*{\cbr_j(\ppp)}
        \sum_{k = 1}^{\nCr} \abs*{\ccr_k(\ppp)}
      }{
        \abs{\omega} \sum_{\ell = 1}^{\nEr} \abs*{\cer_{\ell}(\ppp)}
        + \sum_{i = 1}^{\nAr} \abs*{\car_i(\ppp)}
      }
    }^{2}
    \dif{\omega}
    \difmp{\ppp}.
  \end{equation*}
  Using that
  $\int_{-\infty}^{\infty} \frac{\dif{x}}{{(a \abs{x} + b)}^2} = \frac{2}{a b}$
  for positive $a$ and $b$,
  the above integral becomes equal to the one in~\eqref{eq:eabc-l2-bounded},
  up to scaling by $2$.

  Next, the conditions in~\eqref{eq:ainv-bounded} become
  \begin{align*}
    \esssup_{\ppp \in \psetp} \,
    \esssup_{\omega \in \RR} \,
    \normF*{
      \imag \omega
      \cer_{\ell}(\ppp)
      \myparen*{\imag \omega \Er(\ppp) - \Ar(\ppp)}^{-1}
    }
     & < \infty, \\
    \esssup_{\ppp \in \psetp} \,
    \esssup_{\omega \in \RR} \,
    \normF*{
      \car_i(\ppp)
      \myparen*{\imag \omega \Er(\ppp) - \Ar(\ppp)}^{-1}
    }
     & < \infty,
  \end{align*}
  which simplify to
  \begin{align*}
    \esssup_{\ppp \in \psetp} \,
    \abs*{\cer_{\ell}(\ppp)}
    \esssup_{s \in \imag \RR} \,
    \normF*{s \myparen*{s \Er(\ppp) - \Ar(\ppp)}^{-1}}
     & < \infty, \\
    \esssup_{\ppp \in \psetp} \,
    \abs*{\car_i(\ppp)}
    \esssup_{s \in \imag \RR} \,
    \normF*{\myparen*{s \Er(\ppp) - \Ar(\ppp)}^{-1}}
     & < \infty.
  \end{align*}
  Since $\normF{\cdot}$ and $\normtwo{\cdot}$ are equivalent norms,
  the above conditions are equivalent to~\eqref{eq:linf-bounded}.
\end{proof}
The work~\cite{HunMMS22} used the assumptions that
$\psetp \subset \RRparp$ is compact,
$\measurep$ is a finite Borel measure over $\psetp$,
$\fundef{\cer_{\ell}, \car_i, \cbr_j, \ccr_k}{\psetp}{\RR}$ are continuous,
$\Er(\ppp)$ is invertible and
${\Er(\ppp)}^{-1} \Ar(\ppp)$ has all eigenvalues in the open left half-plane
for all $\ppp \in \psetp$.
Therefore, we see that the assumptions of the earlier work~\cite{HunMMS22} on
$\HtwoLtwo$ approximation are indeed a special case of the ones we derived in
Lemma~\ref{lem:assumptions}.

Now that we established the assumptions of Theorem~\ref{thm:l2-cond-diag} for
parametric \ac{lti} systems,
we are ready to derive the corresponding interpolatory $\HtwoLtwo$-optimality
conditions based on the conditions~\eqref{eq:l2-cond-diag}.
We use $\Er(\ppp) = I$ and all $\hA_i$ being diagonal.
(The result can be extended to parametric diagonal $\Er$;
only the expressions become more involved.)
\begin{thm}\label{thm:h2l2-cond-diag}
  Given the full-order parametric transfer function $H$,
  let $\hH$ in~\eqref{eq:p-rom-tf} be an $\HtwoLtwo$-optimal \ac{diagstrom}
  for $H$ with $\Er(\ppp) = I$ and all $\hA_i$ being diagonal
  in~\eqref{eq:lti-p-sep-form-a}.
  Let $\lambda_\ell(\ppp)$ denote the $\ell$th diagonal entry of $\Ar(\ppp)$.
  Moreover, define
  $c_{\ell}(\ppp) = \Cr(\ppp) e_{\ell}$ and
  $b_{\ell}(\ppp) = \Br(\ppp)\herm e_{\ell}$,
  where $\Br(\ppp)$ and $\Cr(\ppp)$ are as defined
  in~\eqref{eq:lti-p-sep-form-b}.
  Then
  \begin{equation}\label{eq:param-pole-res}
    \hH(s, \ppp)
    = \sum_{\ell = 1}^{\nrom}
    \frac{c_{\ell}(\ppp) b_{\ell}(\ppp)\herm}{s - \lambda_{\ell}(\ppp)}
  \end{equation}
  and
  \begin{subequations}\label{eq:h2l2-cond-diag}
    \begin{gather}
      \label{eq:h2l2-cond-diag-a}
      \begin{aligned}
         &
        \int_{\psetp}
        \ccr_k(\ppp)
        H\myparen*{-\overline{\lambda_\ell(\ppp)}, \ppp}
        b_\ell(\ppp)
        \difmp{\ppp} \\*
         & =
        \int_{\psetp}
        \ccr_k(\ppp)
        \hH\myparen*{-\overline{\lambda_\ell(\ppp)}, \ppp}
        b_\ell(\ppp)
        \difmp{\ppp},
      \end{aligned} \\
      \label{eq:h2l2-cond-diag-b}
      \begin{aligned}
         &
        \int_{\psetp}
        \cbr_j(\ppp)
        c_\ell(\ppp)\herm
        H\myparen*{-\overline{\lambda_\ell(\ppp)}, \ppp}
        \difmp{\ppp} \\*
         & =
        \int_{\psetp}
        \cbr_j(\ppp)
        c_\ell(\ppp)\herm
        \hH\myparen*{-\overline{\lambda_\ell(\ppp)}, \ppp}
        \difmp{\ppp},
      \end{aligned} \\
      \label{eq:h2l2-cond-diag-c}
      \begin{aligned}
         &
        \int_{\psetp}
        \car_i(\ppp)
        c_\ell(\ppp)\herm
        \frac{\partial H}{\partial s}
        \myparen*{-\overline{\lambda_\ell(\ppp)}, \ppp}
        b_\ell(\ppp)
        \difmp{\ppp} \\*
         & =
        \int_{\psetp}
        \car_i(\ppp)
        c_\ell(\ppp)\herm
        \frac{\partial \hH}{\partial s}
        \myparen*{-\overline{\lambda_\ell(\ppp)}, \ppp}
        b_\ell(\ppp)
        \difmp{\ppp},
      \end{aligned}
    \end{gather}
  \end{subequations}
  for $\ell = 1, 2, \ldots, \nrom$,
  $k = 1, 2, \ldots, \nCr$,
  $j = 1, 2, \ldots, \nBr$,
  $i = 1, 2, \ldots \nAr$,
  where $\car_i$, $\cbr_j$, and $\ccr_k$ are as defined
  in~\eqref{eq:lti-p-sep-form}.
\end{thm}
\begin{proof}
  The pole-residue form~\eqref{eq:param-pole-res} follows
  from the general diagonal pole-residue form~\eqref{eq:diag-pole-res},
  with
  $k_{\ell}(s, \ppp) = s - \lambda_{\ell}(\ppp)$,
  $f_{\ell}(s, \ppp) = b_{\ell}(\ppp)$, and
  $g_{\ell}(s, \ppp) = c_{\ell}(\ppp)$.
  Then, with this structure, the optimality conditions~\eqref{eq:h2l2-cond-diag}
  follow from diagonal conditions after applying the Cauchy integral formula.
  For instance, the left-hand side of the right tangential Lagrange
  condition~\eqref{eq:l2-cond-diag-C} becomes
  \begin{align*}
     &
    \int_{\pset}
    \frac{
      \overline{\ccr_k(\pp)}
      \yf(\pp)
      b_{\ell}(\pp)
    }{
      \overline{a_{\ell}(\pp)}
    }
    \difm{\pp}   \\
     & =
    \int_{\psetp}
    \int_{-\infty}^{\infty}
    \frac{
      \ccr_k(\ppp)
      \yf(\imag \omega, \ppp)
      b_{\ell}(\ppp)
    }{
      \overline{\imag \omega - \lambda_{\ell}(\ppp)}
    }
    \dif{\omega}
    \difmp{\ppp} \\
     & =
    \int_{\psetp}
    \int_{-\infty}^{\infty}
    \frac{
      \ccr_k(\ppp)
      \yf(\imag \omega, \ppp)
      b_{\ell}(\ppp)
    }{
      -\imag \omega - \overline{\lambda_{\ell}(\ppp)}
    }
    \dif{\omega}
    \difmp{\ppp} \\
     & =
    \frac{1}{\imag}
    \int_{\psetp}
    \oint_{\imag \RR}
    \frac{
      \ccr_k(\ppp)
      \yf(s, \ppp)
      b_{\ell}(\ppp)
    }{
      -s - \overline{\lambda_{\ell}(\ppp)}
    }
    \dif{s}
    \difmp{\ppp} \\
     & =
    \frac{2 \pi}{\imag}
    \int_{\psetp}
    \ccr_k(\ppp)
    \yf\myparen*{-\overline{\lambda_{\ell}(\ppp)}, \ppp}
    b_{\ell}(\ppp)
    \difmp{\ppp},
  \end{align*}
  which yields~\eqref{eq:h2l2-cond-diag-a}.
  The remaining two
  conditions~\eqref{eq:h2l2-cond-diag-b}--\eqref{eq:h2l2-cond-diag-c}
  follow similarly from~\eqref{eq:l2-cond-diag-B} and~\eqref{eq:l2-cond-diag-A}.
\end{proof}
Recall that $\Htwo$-optimal approximation of \emph{non-parametric} \ac{lti}
systems requires bitangential Hermite interpolation of the \ac{fom} transfer
function $H$ at the mirror images of the reduced-order
poles~\cite{GugAB08,AntBG20}.
We showed in our earlier works~\cite{MliG23a,MliG23b,MliBG23} that bitangential
Hermite interpolation as necessary conditions for optimality extend to many
other $\Htwo$/$\Ltwo$ approximation settings as well.
Even though the $\HtwoLtwo$ optimality conditions~\eqref{eq:h2l2-cond-diag}
derived here have an integral form,
they still have the similar bitangential Hermite interpolation structure as
before.
To arrive at this more familiar form of bitangential Hermite interpolations,
we need to have explicit expressions for the functions
$\car_i, \cbr_j, \ccr_k, \lambda_{\ell}, b_{\ell}, c_{\ell}$.
In the next two sections we focus on such cases.

\section{Parameters in Inputs and Outputs}%
\label{sec:param-io}

The work~\cite{BauBBG11} considered parametric \ac{lti} systems with parameters
only in $\Br$ and $\Cr$;
specifically, the \ac{rom} of the form
\begin{equation}\label{eq:io-ss-rom}
  \Er(\ppp) = I, \ \
  \Ar(\ppp) = \cAr, \ \
  \Br(\ppp) = \cBr_1 + \ppp_1 \cBr_2, \ \
  \Cr(\ppp) = \cCr_1 + \ppp_2 \cCr_2,
\end{equation}
with $\cAr = \mydiag{\lambda_1, \lambda_2, \ldots, \lambda_{\nrom}}$ and
$\psetp = {[0, 1]}^2$.
Therefore,
\begin{gather*}
  \ppp = (\ppp_1, \ppp_2), \
  \nAr = 1, \
  \car_1(\ppp) = 1, \\
  \nBr = 2, \
  \cbr_1(\ppp) = 1, \
  \cbr_2(\ppp) = \ppp_1, \\
  \nCr = 2, \
  \ccr_1(\ppp) = 1, \
  \ccr_2(\ppp) = \ppp_2, \\
  \lambda_{\ell}(\ppp) = \lambda_{\ell}, \
  b_{\ell}(\ppp) = b_{\ell, 1} + \ppp_1 b_{\ell, 2}, \
  c_{\ell}(\ppp) = c_{\ell, 1} + \ppp_2 c_{\ell, 2},
\end{gather*}
where $b_{\ell, i} = \cBr_i\herm e_{\ell}$ and $c_{\ell, i} = \cCr_i e_{\ell}$
for $i = 1, 2$.
In~\cite{BauBBG11}, only \ac{siso} systems were considered.
Here we consider \ac{mimo} systems.
(Further extensions are possible,
see, e.g.,~\cite{KleGW+20}.)

Note that the reduced transfer function is bilinear
in terms of the parameters $\ppp_1$ and $\ppp_2$:
\begin{equation}\label{eq:io-tf-rom}
  \begin{aligned}
    \hH(s, \ppp)
     & =
    \myparen*{\cCr_1 + \ppp_2 \cCr_2}
    \myparen*{s I - \cAr}^{-1}
    \myparen*{\cBr_1 + \ppp_1 \cBr_2} \\*
     & =
    \hH_{11}(s)
    + \ppp_1 \hH_{12}(s)
    + \ppp_2 \hH_{21}(s)
    + \ppp_1 \ppp_2 \hH_{22}(s)
  \end{aligned}
\end{equation}
where
\begin{equation*}
  \hH_{ij}(s) = \cCr_i \myparen*{s I - \cAr}^{-1} \cBr_j, \quad
  i, j \in \{1, 2\}.
\end{equation*}
We assume the same form for the full transfer function
\begin{equation}\label{eq:io-tf-fom}
  H(s, \ppp)
  =
  H_{11}(s)
  + \ppp_1 H_{12}(s)
  + \ppp_2 H_{21}(s)
  + \ppp_1 \ppp_2 H_{22}(s),
\end{equation}
where $H_{ij} \in \Htwo$ for $i, j \in \{1, 2\}$.
However, contrary to~\cite{BauBBG11},
we do not need to assume that $H_{ij}$ has a finite-dimensional state space,
i.e., they can contain non-rational terms.
We only require the \ac{rom} to have a finite-dimensional state space.
Thus, the theory we develop applies not only to \ac{mimo} systems
but also to irrational transfer functions.

Following~\cite{BauBBG11}, we define the auxiliary transfer functions
\begin{subequations}
  \begin{align}
    \label{eq:io-tf-fom-aux}
    \cH(s)
     & =
    \begin{bmatrix}
      H_{11}(s) & H_{12}(s) \\
      H_{21}(s) & H_{22}(s)
    \end{bmatrix}\!, \\
    \label{eq:io-tf-rom-aux}
    \hcH(s)
     & =
    \begin{bmatrix}
      \hH_{11}(s) & \hH_{12}(s) \\
      \hH_{21}(s) & \hH_{22}(s)
    \end{bmatrix}
    =
    \begin{bmatrix}
      \cCr_1 \\
      \cCr_2
    \end{bmatrix}
    \myparen*{s I - \cAr}^{-1}
    \begin{bmatrix}
      \cBr_1 & \cBr_2
    \end{bmatrix}\!.
  \end{align}
\end{subequations}
Note that $H$ and $\hH$ can be obtained from $\cH$ and $\hcH$ via
\begin{subequations}\label{eq:io-aux-revert}
  \begin{align}\label{eq:io-aux-revert-fom}
    H(s, \ppp)
     & =
    \begin{bmatrix}
      I_{\nout} & \ppp_2 I_{\nout}
    \end{bmatrix}
    \cH(s)
    \begin{bmatrix}
      I_{\nin} \\
      \ppp_1 I_{\nin}
    \end{bmatrix}\!, \\
    \hH(s, \ppp)
     & =
    \begin{bmatrix}
      I_{\nout} & \ppp_2 I_{\nout}
    \end{bmatrix}
    \hcH(s)
    \begin{bmatrix}
      I_{\nin} \\
      \ppp_1 I_{\nin}
    \end{bmatrix}\!.
  \end{align}
\end{subequations}
We obtain the following result,
were we take $\measurep$ to be the Lebesgue measure over $\psetp$
(as in~\cite{BauBBG11}).
\begin{thm}\label{thm:h2l2-cond-io-aux}
  Let $H, \hH$ be as in~\eqref{eq:io-tf-fom} and~\eqref{eq:io-ss-rom} and
  $\cH, \hcH$ as in~\eqref{eq:io-tf-fom-aux} and~\eqref{eq:io-tf-rom-aux}.
  Furthermore, let $\hH$ be an $\HtwoLtwo$-optimal \ac{rom} for $H$.
  Define
  \begin{equation}\label{eq:io-aux-b-c}
    \mathfrak{b}_{\ell}
    =
    \begin{bmatrix}
      I_{\nin}             & \frac{1}{2} I_{\nin} \\
      \frac{1}{2} I_{\nin} & \frac{1}{3} I_{\nin}
    \end{bmatrix}
    \begin{bmatrix}
      b_{\ell, 1} \\
      b_{\ell, 2}
    \end{bmatrix},\,
    \mathfrak{c}_{\ell}
    =
    \begin{bmatrix}
      I_{\nout}             & \frac{1}{2} I_{\nout} \\
      \frac{1}{2} I_{\nout} & \frac{1}{3} I_{\nout}
    \end{bmatrix}
    \begin{bmatrix}
      c_{\ell, 1} \\
      c_{\ell, 2}
    \end{bmatrix}\!.
  \end{equation}
  Then for $\ell = 1, 2, \ldots, \nrom$, we have
  \begin{subequations}\label{eq:io-cond}
    \begin{align}
      \label{eq:io-cond-1}
      \cH\myparen*{-\overline{\lambda_{\ell}}}
      \mathfrak{b}_{\ell}
       & =
      \hcH\myparen*{-\overline{\lambda_{\ell}}}
      \mathfrak{b}_{\ell},                       \\
      \label{eq:io-cond-2}
      \mathfrak{c}_{\ell}\herm
      \cH\myparen*{-\overline{\lambda_{\ell}}}
       & =
      \mathfrak{c}_{\ell}\herm
      \hcH\myparen*{-\overline{\lambda_{\ell}}}, \\
      \label{eq:io-cond-3}
      \mathfrak{c}_{\ell}\herm
      \cH'\myparen*{-\overline{\lambda_{\ell}}}
      \mathfrak{b}_{\ell}
       & =
      \mathfrak{c}_{\ell}\herm
      \hcH'\myparen*{-\overline{\lambda_{\ell}}}
      \mathfrak{b}_{\ell}.
    \end{align}
  \end{subequations}
\end{thm}
\begin{proof}
  Based on the conditions in~\eqref{eq:h2l2-cond-diag},
  we need to compute the integrals
  \begin{gather*}
    \int_{\psetp}
    H\myparen*{-\overline{\lambda_{\ell}}, \ppp}
    b_\ell(\ppp)
    \dif{\ppp},\
    \int_{\psetp}
    \ppp_2
    H\myparen*{-\overline{\lambda_{\ell}}, \ppp}
    b_\ell(\ppp)
    \dif{\ppp}, \\
    \int_{\psetp}
    c_\ell(\ppp)\herm
    H\myparen*{-\overline{\lambda_{\ell}}, \ppp}
    \dif{\ppp},\
    \int_{\psetp}
    \ppp_1
    c_\ell(\ppp)\herm
    H\myparen*{-\overline{\lambda_{\ell}}, \ppp}
    \dif{\ppp}, \\
    \int_{\psetp}
    c_\ell(\ppp)\herm
    \frac{\partial H}{\partial s}
    \myparen*{-\overline{\lambda_{\ell}}, \ppp}
    b_\ell(\ppp)
    \dif{\ppp},
  \end{gather*}
  and similarly with $\hH$.
  Starting with the first, we find that
  \begin{align*}
     &
    \int_{\psetp}
    H\myparen*{-\overline{\lambda_{\ell}}, \ppp}
    b_\ell(\ppp)
    \dif{\ppp}                                                            \\
     & =
    \int_{\psetp}
    \Bigl(
    H_{11}\myparen*{-\overline{\lambda_{\ell}}}
    + \ppp_1 H_{12}\myparen*{-\overline{\lambda_{\ell}}}
    + \ppp_2 H_{21}\myparen*{-\overline{\lambda_{\ell}}}                  \\*
     & \qquad\quad
    + \ppp_1 \ppp_2 H_{22}\myparen*{-\overline{\lambda_{\ell}}}
    \Bigr)
    \myparen*{b_{\ell, 1} + \ppp_1 b_{\ell, 2}}
    \dif{\ppp}                                                            \\
     & =
    H_{11}\myparen*{-\overline{\lambda_{\ell}}} b_{\ell, 1}
    + \frac{1}{2} H_{12}\myparen*{-\overline{\lambda_{\ell}}} b_{\ell, 1} \\*
     & \qquad
    + \frac{1}{2} H_{21}\myparen*{-\overline{\lambda_{\ell}}} b_{\ell, 1}
    + \frac{1}{4} H_{22}\myparen*{-\overline{\lambda_{\ell}}} b_{\ell, 1} \\*
     & \qquad
    + \frac{1}{2} H_{11}\myparen*{-\overline{\lambda_{\ell}}} b_{\ell, 2}
    + \frac{1}{3} H_{12}\myparen*{-\overline{\lambda_{\ell}}} b_{\ell, 2} \\*
     & \qquad
    + \frac{1}{4} H_{21}\myparen*{-\overline{\lambda_{\ell}}} b_{\ell, 2}
    + \frac{1}{6} H_{22}\myparen*{-\overline{\lambda_{\ell}}} b_{\ell, 2} \\
     & =
    \begin{bmatrix}
      I_{\nout} & \frac{1}{2} I_{\nout}
    \end{bmatrix}
    \cH\myparen*{-\overline{\lambda_{\ell}}}
    \begin{bmatrix}
      I_{\nin}             & \frac{1}{2} I_{\nin} \\
      \frac{1}{2} I_{\nin} & \frac{1}{3} I_{\nin}
    \end{bmatrix}
    \begin{bmatrix}
      b_{\ell, 1} \\
      b_{\ell, 2}
    \end{bmatrix}
  \end{align*}
  where we used~\eqref{eq:io-tf-fom} and~\eqref{eq:io-aux-revert-fom}.
  The second integral becomes
  \begin{align*}
     &
    \int_{\psetp}
    \ppp_2
    H\myparen*{-\overline{\lambda_{\ell}}, \ppp}
    b_\ell(\ppp)
    \dif{\ppp} \\
     & =
    \begin{bmatrix}
      \frac{1}{2} I_{\nout} & \frac{1}{3} I_{\nout}
    \end{bmatrix}
    \cH\myparen*{-\overline{\lambda_{\ell}}}
    \begin{bmatrix}
      I_{\nin}             & \frac{1}{2} I_{\nin} \\
      \frac{1}{2} I_{\nin} & \frac{1}{3} I_{\nin}
    \end{bmatrix}
    \begin{bmatrix}
      b_{\ell, 1} \\
      b_{\ell, 2}
    \end{bmatrix}.
  \end{align*}
  Stacking these two vertically (and using~\eqref{eq:io-aux-b-c}) gives us the
  left-hand side in the right Lagrange tangential
  condition~\eqref{eq:io-cond-1}.
  The other conditions follow similarly.
\end{proof}
Therefore, for this special case of~\eqref{eq:io-ss-rom},
we obtain more familiar bitangential Hermite interpolation.
More specifically, $\HtwoLtwo$-optimal reduced-order modeling of $H$ with $\hH$
in~\eqref{eq:io-ss-rom} is equivalent to a \emph{weighted} $\Htwo$-optimal
reduced-order modeling for $\cH$ with $\hcH$.
Thus, our general framework in Theorem~\ref{thm:h2l2-cond-diag} not only recovers
the results from~\cite{BauBBG11} but also extends it to \ac{mimo} systems and
eliminates the need for the \ac{fom} to have a rational transfer function.

Interpolatory conditions of Theorem~\ref{thm:h2l2-cond-io-aux} are in terms of the
parametric function $\hcH$, not the original parametric transfer function $H$.
The next result gives explicit interpolatory conditions in terms of $H$
for \ac{siso} systems.
\begin{cor}
  Let the assumptions in Theorem~\ref{thm:h2l2-cond-io-aux} hold.
  Furthermore, let $\nin = \nout = 1$.
  If $\mathfrak{b}_{\ell, 1} \neq 0$ and $\mathfrak{c}_{\ell, 1} \neq 0$, then
  \begin{align*}
    H\myparen*{
      -\overline{\lambda_{\ell}},
      \frac{\mathfrak{b}_{\ell, 2}}{\mathfrak{b}_{\ell, 1}},
      \ppp_2
    }
     & =
    \hH\myparen*{
      -\overline{\lambda_{\ell}},
      \frac{\mathfrak{b}_{\ell, 2}}{\mathfrak{b}_{\ell, 1}},
      \ppp_2
    },   \\
    \partial_{\ppp_2} H\myparen*{
      -\overline{\lambda_{\ell}},
      \frac{\mathfrak{b}_{\ell, 2}}{\mathfrak{b}_{\ell, 1}},
      \ppp_2
    }
     & =
    \partial_{\ppp_2} \hH\myparen*{
      -\overline{\lambda_{\ell}},
      \frac{\mathfrak{b}_{\ell, 2}}{\mathfrak{b}_{\ell, 1}},
      \ppp_2
    },   \\
    H\myparen*{
      -\overline{\lambda_{\ell}},
      \ppp_1,
      \frac{\overline{\mathfrak{c}_{\ell, 2}}}{\overline{\mathfrak{c}_{\ell, 1}}}
    }
     & =
    \hH\myparen*{
      -\overline{\lambda_{\ell}},
      \ppp_1,
      \frac{\overline{\mathfrak{c}_{\ell, 2}}}{\overline{\mathfrak{c}_{\ell, 1}}}
    },   \\
    \partial_{\ppp_1} H\myparen*{
      -\overline{\lambda_{\ell}},
      \ppp_1,
      \frac{\overline{\mathfrak{c}_{\ell, 2}}}{\overline{\mathfrak{c}_{\ell, 1}}}
    }
     & =
    \partial_{\ppp_1} \hH\myparen*{
      -\overline{\lambda_{\ell}},
      \ppp_1,
      \frac{\overline{\mathfrak{c}_{\ell, 2}}}{\overline{\mathfrak{c}_{\ell, 1}}}
    },   \\
    H'\myparen*{
      -\overline{\lambda_{\ell}},
      \frac{\mathfrak{b}_{\ell, 2}}{\mathfrak{b}_{\ell, 1}},
      \frac{\overline{\mathfrak{c}_{\ell, 2}}}{\overline{\mathfrak{c}_{\ell, 1}}}
    }
     & =
    \hH'\myparen*{
      -\overline{\lambda_{\ell}},
      \frac{\mathfrak{b}_{\ell, 2}}{\mathfrak{b}_{\ell, 1}},
      \frac{\overline{\mathfrak{c}_{\ell, 2}}}{\overline{\mathfrak{c}_{\ell, 1}}}
    },
  \end{align*}
  for all $\ppp_1, \ppp_2 \in \CC$ and $\ell = 1, 2, \ldots, \nrom$.
\end{cor}
\begin{proof}
  The proof follows directly from~\eqref{eq:io-cond}
  using~\eqref{eq:io-aux-revert} and
  \begin{equation*}
    \partial_{\ppp_2} H(s, \ppp)
    =
    \begin{bmatrix}
      0 & 1
    \end{bmatrix}
    \cH(s)
    \begin{bmatrix}
      1 \\
      \ppp_1
    \end{bmatrix}
  \end{equation*}
  and similar expressions for $\partial_{\ppp_2} \hH$, $\partial_{\ppp_1} H$,
  and $\partial_{\ppp_1} \hH$.
\end{proof}
This states that $\HtwoLtwo$-optimality for the full-order and reduced-order
structure in~\eqref{eq:io-tf-fom} and~\eqref{eq:io-tf-rom} requires that,
in the parameter space, interpolation is enforced over lines instead of just a
finite number of points
(as generically done for parametric systems).

\section{Parameter in Dynamics}%
\label{sec:param-dyn}

In the previous section,
we considered a special case where the parametric dependence was only in $\Br$
and $\Cr$.
Now, we consider the case where $\psetp = [a, b] \subset \RR$, $a < b$, and
the \ac{fom} and \ac{rom} have the form,
respectively,
\begin{equation}\label{eq:dyn-ss-fom}
  \Ef(\ppp) = I, \ \
  \Af(\ppp) = \cAf_1 + \ppp \cAf_2, \ \
  \Bf(\ppp) = \cBf, \ \
  \Cf(\ppp) = \cCf
\end{equation}
and
\begin{equation}\label{eq:dyn-ss-rom}
  \Er(\ppp) = I, \ \
  \Ar(\ppp) = \cAr_1 + \ppp \cAr_2, \ \
  \Br(\ppp) = \cBr, \ \
  \Cr(\ppp) = \cCr
\end{equation}
with
\begin{align*}
  \cAf_k
   & =
  \mydiag{\nu_{k, 1}, \nu_{k, 2}, \ldots, \nu_{k, \nfom}}
  \ \text{ and} \\
  \cAr_k
   & =
  \mydiag{\lambda_{k, 1}, \lambda_{k, 2}, \ldots, \lambda_{k, \nrom}},
\end{align*}
for $k = 1, 2$.
In other words, we are assuming the parametric dependencies appear only in the
dynamics matrices $\Af$ and $\Ar$ and they are both composed of only two terms
which are simultaneously diagonalizable.
With these parametric forms,
the full-order and reduced-order transfer functions have the pole-residue forms
\begin{equation}\label{eq:dyn-tf}
  H(s, \ppp)
  =
  \sum_{i = 1}^{\nfom}
  \frac{\Phi_i}{s - \nu_i(\ppp)}, \quad
  \hH(s, \ppp)
  =
  \sum_{i = 1}^{\nrom}
  \frac{c_i b_i\herm}{s - \lambda_i(\ppp)},
\end{equation}
where
\begin{equation}\label{eq:nu-lambda}
  \nu_i(\ppp) = \nu_{1, i} + \ppp \nu_{2, i}
  \quad \text{and} \quad
  \lambda_i(\ppp) = \lambda_{1, i} + \ppp \lambda_{2, i}.
\end{equation}
Let $\measurep$ be the Lebesgue measure over $[a, b]$.
Furthermore,
for any $\sigma_a, \sigma_b \in \CC_-$,
define $\fundef{f_{\sigma_a, \sigma_b}}{\CC_+^2}{\CC}$ as
\begin{equation}\label{eq:f}
  f_{\sigma_a, \sigma_b}(s_a, s_b)
  =
  \frac{b - a}{(s_b - \sigma_b) - (s_a - \sigma_a)}
  \ln\myparen*{\frac{s_b - \sigma_b}{s_a - \sigma_a}}.
\end{equation}
Additionally, define the functions $\fundef{G, \hG}{\CC_+^2}{\CCoi}$ by
\begin{subequations}\label{eq:dyn-G}
  \begin{align}
    \label{eq:dyn-G-fom}
    G(s_a, s_b)
     & =
    \sum_{i = 1}^{\nfom}
    f_{\nu_i(a), \nu_i(b)}(s_a, s_b)
    \Phi_i
    \quad \text{and} \\
    \label{eq:dyn-G-rom}
    \hG(s_a, s_b)
     & =
    \sum_{i = 1}^{\nrom}
    f_{\lambda_i(a), \lambda_i(b)}(s_a, s_b)
    c_i b_i\herm.
  \end{align}
\end{subequations}
Note that $G$ and $\hG$ depend on the pole-residue forms~\eqref{eq:dyn-tf}
of $H$ and $\hH$, respectively.
Thus, one can consider $G$ as the full-order \emph{modified} function and
$\hG$ the reduced-order one.
Based on this setup,
we are ready to state the interpolatory optimality conditions in this setting.
\begin{thm}\label{thm:h2l2-cond-dyn}
  Let $H$ and $\hH$ be as given in~\eqref{eq:dyn-tf} and let $G$ and $\hG$
  be as defined in~\eqref{eq:dyn-G}.
  If $\hH$ is an $\HtwoLtwo$-optimal \ac{diagstrom} for $H$,
  then
  \begin{subequations}\label{eq:p-simple-cond}
    \begin{align}
      \label{eq:p-simple-cond-1}
      G\myparen*{
        -\overline{\lambda_i(a)},
        -\overline{\lambda_i(b)}
      }
      b_i
       & =
      \hG\myparen*{
        -\overline{\lambda_i(a)},
        -\overline{\lambda_i(b)}
      }
      b_i, \\
      \label{eq:p-simple-cond-2}
      c_i\herm
      G\myparen*{
        -\overline{\lambda_i(a)},
        -\overline{\lambda_i(b)}
      }
       & =
      c_i\herm
      \hG\myparen*{
        -\overline{\lambda_i(a)},
        -\overline{\lambda_i(b)}
      },   \\
      \label{eq:p-simple-cond-3}
      c_i\herm
      \frac{\partial G}{\partial s_a}\myparen*{
        -\overline{\lambda_i(a)},
        -\overline{\lambda_i(b)}
      }
      b_i
       & =
      c_i\herm
      \frac{\partial \hG}{\partial s_a}\myparen*{
        -\overline{\lambda_i(a)},
        -\overline{\lambda_i(b)}
      }
      b_i, \\
      \label{eq:p-simple-cond-4}
      c_i\herm
      \frac{\partial G}{\partial s_b}\myparen*{
        -\overline{\lambda_i(a)},
        -\overline{\lambda_i(b)}
      }
      b_i
       & =
      c_i\herm
      \frac{\partial \hG}{\partial s_b}\myparen*{
        -\overline{\lambda_i(a)},
        -\overline{\lambda_i(b)}
      }
      b_i,
    \end{align}
  \end{subequations}
  for $i = 1, 2, \ldots, \nrom$.
\end{thm}
\begin{proof}
  Due to the special form of the \ac{rom} in~\eqref{eq:dyn-ss-rom} and
  its transfer function $\hH$ in~\eqref{eq:dyn-tf},
  the quantities $b_{\ell}, c_{\ell}, \ccr_k, \cbr_j$
  in~\eqref{eq:param-pole-res} and~\eqref{eq:h2l2-cond-diag} in
  Theorem~\ref{thm:h2l2-cond-diag} are parameter-independent.
  Thus, to analyze the first Lagrange conditions~\eqref{eq:h2l2-cond-diag-a}
  and~\eqref{eq:h2l2-cond-diag-b}
  for the special case~\eqref{eq:dyn-ss-fom} and~\eqref{eq:dyn-ss-rom},
  it is enough to focus on the integrals
  \begin{equation*}
    \int_{\psetp}
    H\myparen*{-\overline{\lambda_i(\ppp)}, \ppp}
    \difmp{\ppp}\quad\mbox{and}\quad\int_{\psetp}
    \hH\myparen*{-\overline{\lambda_i(\ppp)}, \ppp}
    \difmp{\ppp}.
  \end{equation*}
  We start with the first integral involving $H$.
  Using the pole-residue form of $H$ from~\eqref{eq:dyn-tf} and
  the expression for $\lambda_i(\ppp)$ from~\eqref{eq:nu-lambda},
  we obtain
  \begin{align*}
     &
    \int_{\psetp}
    H\myparen*{-\overline{\lambda_i(\ppp)}, \ppp}
    \difmp{\ppp}
    =
    \int_a^b
    H\myparen*{-\overline{\lambda_i(\ppp)}, \ppp}
    \dif{\ppp} \\
     & =
    \int_a^b
    \sum_{j = 1}^{\nfom}
    \frac{\Phi_j}{-\overline{\lambda_i(\ppp)} - \nu_{1, j} - \ppp \nu_{2, j}}
    \dif{\ppp} \\
     & =
    \sum_{j = 1}^{\nfom}
    \int_a^b
    \frac{\Phi_j}{
      -\overline{\lambda_{1, i}}
      - \ppp \overline{\lambda_{2, i}}
      - \nu_{1, j}
      - \ppp \nu_{2, j}
    }
    \dif{\ppp} \\
     & =
    \sum_{j = 1}^{\nfom}
    \int_a^b
    \frac{\Phi_j}{
      -\overline{\lambda_{1, i}}
      - \nu_{1, j}
      + \ppp
      \myparen*{
        -\overline{\lambda_{2, i}}
        - \nu_{2, j}
      }
    }
    \dif{\ppp}.
  \end{align*}
  Then integrating the last equality gives
  \begin{align*}
     &
    \int_{\psetp}
    H\myparen*{-\overline{\lambda_i(\ppp)}, \ppp}
    \difmp{\ppp} \\
     & =
    \sum_{j = 1}^{\nfom}
    \frac{\Phi_j}{
      -\overline{\lambda_{2, i}}
      - \nu_{2, j}
    }
    \ln\myparen*{
      \frac{
        -\overline{\lambda_{1, i}}
        - \nu_{1, j}
        + b
        \myparen*{
          -\overline{\lambda_{2, i}}
          - \nu_{2, j}
        }
      }{
        -\overline{\lambda_{1, i}}
        - \nu_{1, j}
        + a
        \myparen*{
          -\overline{\lambda_{2, i}}
          - \nu_{2, j}
        }
      }
    }            \\
     & =
    \sum_{j = 1}^{\nfom}
    \frac{\Phi_j (b - a)}{
      \myparen*{
        -\overline{\lambda_i(b)}
        - \nu_j(b)
      }
      - \myparen*{
        -\overline{\lambda_i(a)}
        - \nu_j(a)
      }
    }
    \ln\myparen*{
      \frac{
        -\overline{\lambda_i(b)}
        - \nu_j(b)
      }{
        -\overline{\lambda_i(a)}
        - \nu_j(a)
      }
    }            \\
     & =
    G\myparen*{-\overline{\lambda_i(a)}, -\overline{\lambda_i(b)}}.
  \end{align*}
  Similarly, one can show that
  \begin{align*}
    \int_{\psetp}
    \hH\myparen*{-\overline{\lambda_i(\ppp)}, \ppp}
    \difmp{\ppp}
     & =
    \hG\myparen*{-\overline{\lambda_i(a)}, -\overline{\lambda_i(b)}}.
  \end{align*}
  Therefore, the first two optimality conditions~\eqref{eq:h2l2-cond-diag-a}
  and~\eqref{eq:h2l2-cond-diag-b}
  in Theorem~\ref{thm:h2l2-cond-diag} lead to the interpolatory
  conditions~\eqref{eq:p-simple-cond-1} and~\eqref{eq:p-simple-cond-2}.

  To derive the remaining two conditions~\eqref{eq:p-simple-cond-3}
  and~\eqref{eq:p-simple-cond-4} from~\eqref{eq:h2l2-cond-diag-c},
  we now consider the integrals
  \begin{equation}\label{eq:two-integrals}
    \int_{\psetp}
    \ppp^k
    \frac{\partial H}{\partial s}
    \myparen*{-\overline{\lambda_i(\ppp)}, \ppp}
    \difmp{\ppp}, \quad
    k = 0, 1.
  \end{equation}
  We will need the expressions for the partial derivatives of $G$.
  It directly follows from~\eqref{eq:dyn-G-fom} that
  \begin{equation}\label{eq:dGds}
    \frac{\partial G}{\partial s_a}(s_a, s_b)
    =
    \sum_{i = 1}^{\nfom}
    \frac{\partial f_{\nu_i(a), \nu_i(b)}}{\partial s_a}(s_a, s_b)
    \Phi_i.
  \end{equation}
  Similar expressions hold for
  $\frac{\partial G}{\partial s_b}(s_a, s_b)$,
  $\frac{\partial \hG}{\partial s_a}(s_a, s_b)$, and
  $\frac{\partial \hG}{\partial s_b}(s_a, s_b)$
  as well.
  Thus, to compute these partial derivatives,
  we simply focus on $f_{\sigma_a, \sigma_b}$ and obtain,
  via direct differentiation of~\eqref{eq:f}, that
  \begin{align*}
     &
    \frac{\partial f_{\sigma_a, \sigma_b}}{\partial s_a}(s_a, s_b) \\
     & =
    \frac{b - a}{{((s_b - \sigma_b) - (s_a - \sigma_a))}^2}
    \ln\myparen*{\frac{s_b - \sigma_b}{s_a - \sigma_a}}            \\*
     & \quad
    - \frac{b - a}{(s_b - \sigma_b) - (s_a - \sigma_a)}
    \cdot \frac{1}{s_a - \sigma_a}, \quad \text{and}               \\
     &
    \frac{\partial f_{\sigma_a, \sigma_b}}{\partial s_b}(s_a, s_b) \\
     & =
    -\frac{b - a}{{((s_b - \sigma_b) - (s_a - \sigma_a))}^2}
    \ln\myparen*{\frac{s_b - \sigma_b}{s_a - \sigma_a}}            \\*
     & \quad
    + \frac{b - a}{(s_b - \sigma_b) - (s_a - \sigma_a)}
    \cdot \frac{1}{s_b - \sigma_b}.
  \end{align*}
  Using the pole-residue form of $H$ from~\eqref{eq:dyn-tf}
  in the first integral in~\eqref{eq:two-integrals} gives
  \begin{align*}
     &
    \int_{\psetp}
    \frac{\partial H}{\partial s}
    \myparen*{-\overline{\lambda_i(\ppp)}, \ppp}
    \difmp{\ppp}
    =
    \int_a^b
    \frac{\partial H}{\partial s}
    \myparen*{-\overline{\lambda_i(\ppp)}, \ppp}
    \dif{\ppp}      \\
     & =
    \int_a^b
    \sum_{j = 1}^{\nfom}
    \frac{-\Phi_j}{\myparen*{
        -\overline{\lambda_i(\ppp)} - \nu_{1, j} - \ppp \nu_{2, j}
      }^2}
    \dif{\ppp}      \\
     & =
    \sum_{j = 1}^{\nfom}
    \int_a^b
    \frac{-\Phi_j}{
      \myparen*{
        -\overline{\lambda_{1, i}}
        - \nu_{1, j}
        + \ppp
        \myparen*{
          -\overline{\lambda_{2, i}}
          - \nu_{2, j}
        }
      }^2
    }
    \dif{\ppp}      \\
     & =
    \sum_{j = 1}^{\nfom}
    \frac{-\Phi_j}{
      -\overline{\lambda_{2, i}}
      - \nu_{2, j}
    }
    \Biggl(
    \frac{1}{
      -\overline{\lambda_{1, i}}
      - \nu_{1, j}
      + b
      \myparen*{
        -\overline{\lambda_{2, i}}
        - \nu_{2, j}
      }
    }               \\*
     & \qquad\qquad
    - \frac{1}{
      -\overline{\lambda_{1, i}}
      - \nu_{1, j}
      + a
      \myparen*{
        -\overline{\lambda_{2, i}}
        - \nu_{2, j}
      }
    }
    \Biggr).
  \end{align*}
  After various algebraic manipulations to replace
  $\lambda_{k, i}$ and $\nu_{k, j}$ by
  $\lambda_i(\cdot)$ and $\nu_j(\cdot)$ and using~\eqref{eq:dGds},
  we obtain
  \begin{align*}
     &
    \int_{\psetp}
    \frac{\partial H}{\partial s}
    \myparen*{-\overline{\lambda_i(\ppp)}, \ppp}
    \difmp{\ppp}          \\
     & =
    \sum_{j = 1}^{\nfom}
    \frac{-\Phi_j (b - a)}{
      \myparen*{
        -\overline{\lambda_i(b)}
        - \nu_j(b)
      }
      - \myparen*{
        -\overline{\lambda_i(a)}
        - \nu_j(a)
      }
    }                     \\*
     & \qquad\quad \times
    \myparen*{
      \frac{1}{
        -\overline{\lambda_i(b)}
        - \nu_j(b)
      }
      - \frac{1}{
        -\overline{\lambda_i(a)}
        - \nu_j(a)
      }
    }                     \\
     & =
    -(b - a)
    \myparen*{
      \frac{\partial G}{\partial s_a}
      \myparen*{-\overline{\lambda_i(a)}, -\overline{\lambda_i(b)}}
      +
      \frac{\partial G}{\partial s_b}
      \myparen*{-\overline{\lambda_i(a)}, -\overline{\lambda_i(b)}}
    }.
  \end{align*}
  Following the same derivations,
  one obtains similar expressions involving $\hH$ and $\hG$,
  which shows that
  \begin{equation}\label{eq:GaGb1}
    \begin{aligned}
       &
      c_i\herm
      \myparen*{
        \frac{\partial G}{\partial s_a}
        \myparen*{-\overline{\lambda_i(a)}, -\overline{\lambda_i(b)}}
        +
        \frac{\partial G}{\partial s_b}
        \myparen*{-\overline{\lambda_i(a)}, -\overline{\lambda_i(b)}}
      }
      b_i  \\*
       & =
      c_i\herm
      \myparen*{
        \frac{\partial \hG}{\partial s_a}
        \myparen*{-\overline{\lambda_i(a)}, -\overline{\lambda_i(b)}}
        +
        \frac{\partial \hG}{\partial s_b}
        \myparen*{-\overline{\lambda_i(a)}, -\overline{\lambda_i(b)}}
      }
      b_i.
    \end{aligned}
  \end{equation}
  Focusing on the second integral in~\eqref{eq:two-integrals}, we find
  \begin{align*}
     &
    \int_{\psetp}
    \ppp
    \frac{\partial H}{\partial s}
    \myparen*{-\overline{\lambda_i(\ppp)}, \ppp}
    \difmp{\ppp}
    =
    \int_a^b
    \ppp
    \frac{\partial H}{\partial s}
    \myparen*{-\overline{\lambda_i(\ppp)}, \ppp}
    \dif{\ppp}      \\
     & =
    \sum_{j = 1}^{\nfom}
    \int_a^b
    \frac{-\ppp \Phi_j}{
      \myparen*{
        -\overline{\lambda_{1, i}}
        - \nu_{1, j}
        + \ppp
        \myparen*{
          -\overline{\lambda_{2, i}}
          - \nu_{2, j}
        }
      }^2
    }
    \dif{\ppp}      \\
     & =
    \sum_{j = 1}^{\nfom}
    \frac{-\Phi_j}{
      \myparen*{
        -\overline{\lambda_{2, i}}
        - \nu_{2, j}
      }^2
    }
    \Biggl(
    \frac{
      -\overline{\lambda_{1, i}}
      - \nu_{1, j}
    }{
      -\overline{\lambda_i(b)}
      - \nu_j(b)
    }               \\*
     & \qquad\qquad
    - \frac{
      -\overline{\lambda_{1, i}}
      - \nu_{1, j}
    }{
      -\overline{\lambda_i(a)}
      - \nu_j(a)
    }
    +
    \ln\myparen*{
      \frac{
        -\overline{\lambda_i(b)}
        - \nu_j(b)
      }{
        -\overline{\lambda_i(a)}
        - \nu_j(a)
      }
    }
    \Biggr).
  \end{align*}
  After more tedious algebraic manipulations, we obtain
  \begin{align*}
     &
    \int_{\psetp}
    \ppp
    \frac{\partial H}{\partial s}
    \myparen*{-\overline{\lambda_i(\ppp)}, \ppp}
    \difmp{\ppp}           \\
     & =
    \sum_{j = 1}^{\nfom}
    \frac{\Phi_j (b - a)}{
      \myparen*{
        -\overline{\lambda_i(b)}
        - \nu_j(b)
      }
      - \myparen*{
        -\overline{\lambda_i(a)}
        - \nu_j(a)
      }
    }                      \\*
     & \qquad\quad \times
    \myparen*{
      \frac{b}{
        -\overline{\lambda_i(b)}
        - \nu_j(b)
      }
      - \frac{a}{
        -\overline{\lambda_i(a)}
        - \nu_j(a)
      }
    }                      \\*
     & \quad +
    \sum_{j = 1}^{\nfom}
    \frac{-\Phi_j {(b - a)}^2}{
      \myparen*{
        \myparen*{
          -\overline{\lambda_i(b)}
          - \nu_j(b)
        }
        - \myparen*{
          -\overline{\lambda_i(a)}
          - \nu_j(a)
        }
      }^2
    }                      \\*
     & \qquad\qquad \times
    \ln\myparen*{
      \frac{
        -\overline{\lambda_i(b)}
        - \nu_j(b)
      }{
        -\overline{\lambda_i(a)}
        - \nu_j(a)
      }
    }                      \\
     & =
    a
    \frac{\partial G}{\partial s_a}
    \myparen*{-\overline{\lambda_i(a)}, -\overline{\lambda_i(b)}}
    +
    b
    \frac{\partial G}{\partial s_b}
    \myparen*{-\overline{\lambda_i(a)}, -\overline{\lambda_i(b)}}.
  \end{align*}
  As before, with similar expressions involving $\hH$ and $\hG$, we obtain
  \begin{equation}\label{eq:GaGb2}
    \begin{aligned}
       &
      c_i\herm
      \myparen*{
        a
        \frac{\partial G}{\partial s_a}
        \myparen*{-\overline{\lambda_i(a)}, -\overline{\lambda_i(b)}}
        +
        b
        \frac{\partial G}{\partial s_b}
        \myparen*{-\overline{\lambda_i(a)}, -\overline{\lambda_i(b)}}
      }
      b_i  \\*
       & =
      c_i\herm
      \myparen*{
        a
        \frac{\partial \hG}{\partial s_a}
        \myparen*{-\overline{\lambda_i(a)}, -\overline{\lambda_i(b)}}
        +
        b
        \frac{\partial \hG}{\partial s_b}
        \myparen*{-\overline{\lambda_i(a)}, -\overline{\lambda_i(b)}}
      }
      b_i.
    \end{aligned}
  \end{equation}
  Then~\eqref{eq:GaGb1} and~\eqref{eq:GaGb2} give the last two optimality
  conditions~\eqref{eq:p-simple-cond-3} and~\eqref{eq:p-simple-cond-4},
  thus concluding the proof.
\end{proof}
Theorem~\ref{thm:h2l2-cond-dyn} proves that
for this class of parametric \ac{lti} systems,
bitangential Hermite interpolation,
once again,
forms the foundation of the $\Ltwo$-optimal approximation.
The interpolation is based on a modified, two-variable transfer function $G$,
and has to be enforced at the reflected boundary values of the poles.
This is the first such result for parametric \ac{lti} systems
where the system poles vary with the parameters.
Therefore, we have extended the classical bitangential Hermite interpolation
conditions from non-parametric $\Htwo$-optimal approximation to parametric
$\HtwoLtwo$-optimal approximation.

\section{Numerical Experiments}

In this section, we numerically demonstrate \Cref{thm:h2l2-cond-dyn} on two
numerical examples.
In both cases, we start with \iac{fom} of the form in~\eqref{eq:dyn-ss-fom} and
numerically find a locally $\HtwoLtwo$-optimal \ac{rom} of the from
in~\eqref{eq:dyn-ss-rom} via the BFGS method implemented in
SciPy~\cite{VirGOetal20}.
The gradients are computed based on the expressions in~\cite{HunMMS22} and using
SciPy's numerical quadrature.
The $\Htwo$ norms are computed using pyMOR~\cite{MliRS21}.
Then Theorem~5.1 is verified by computing the relative errors
in~\eqref{eq:p-simple-cond}, i.e.,
computing the absolute difference between the left and the right-hand sides and
dividing by the absolute value of the left-hand side
(since we focus on \ac{siso} systems in the numerical examples,
there is only one Lagrange interpolation condition).
In other words, in both cases we show that optimal \acp{diagstrom} satisfy the
developed interpolatory optimal conditions.

The code to reproduce the results is available at~\cite{Mli24b}.

\subsection{Synthetic Parametric Model}

We first consider a variant of the synthetic parametric model from MOR
Wiki~\cite{morwiki_synth_pmodel}.
In particular, we consider \iac{fom} of order $6$ (instead of $100$), i.e.,
$\Ef(\pp) = I$,
\begin{gather*}
  \Af(\pp) =
  \begin{bmatrix}
    -10 \pp & 10                                            \\
    -10     & -10 \pp                                       \\
            &         & -30 \pp & 30                        \\
            &         & -30     & -30 \pp                   \\
            &         &         &         & 50 \pp & 50     \\
            &         &         &         & -50    & 50 \pp
  \end{bmatrix}\!, \\
  \Bf(\pp) =
  \begin{bmatrix}
    2 \\
    0 \\
    2 \\
    0 \\
    2 \\
    0
  \end{bmatrix}\!, \quad
  \Cf(\pp) =
  \begin{bmatrix}
    1 & 0 & 1 & 0 & 1 & 0
  \end{bmatrix}\!,
\end{gather*}
over the parameter space $\pset = [\frac{1}{50}, 1]$.
To find a locally $\HtwoLtwo$-optimal \ac{rom} of order $4$,
we initialize the BFGS-based minimization with the \ac{fom} truncated to the
first $4$ states.
To obtain \iac{diagstrom},
we enforce the complex diagonal structure with $\Er(\pp) = I$,
\begin{gather*}
  \Ar(\pp) =
  \begin{bmatrix}
    x_1 + \pp x_2  & x_3 + \pp x_4                                  \\
    -x_3 - \pp x_4 & x_1 + \pp x_2                                  \\
                   &               & x_5 + \pp x_6  & x_7 + \pp x_8 \\
                   &               & -x_7 - \pp x_8 & x_5 + \pp x_6
  \end{bmatrix}\!, \\
  \Br(\pp) =
  \begin{bmatrix}
    2 \\
    0 \\
    2 \\
    0
  \end{bmatrix}\!, \quad \text{and} \quad
  \Cr(\pp) =
  \begin{bmatrix}
    x_9 & x_{10} & x_{11} & x_{12}
  \end{bmatrix}\!.
\end{gather*}
Upon the convergence of BFGS,
we obtain
\begin{align*}
  \Ar(\pp)
   & =
  \mydiag*{\Ar_1(\pp), \Ar_2(\pp)},               \\
  \Ar_1(\pp)
   & =
  \begin{bmatrix}
    -7.0213 \times 10^{-3} & 9.9975                 \\
    -9.9975                & -7.0213 \times 10^{-3}
  \end{bmatrix} \\*
   & \quad
  + \pp
  \begin{bmatrix}
    -11.014  & 0.24074 \\
    -0.24074 & -11.014
  \end{bmatrix}                              \\
  \Ar_2(\pp)
   & =
  \begin{bmatrix}
    -1.6795 - 39.184 \pp  & 29.261 + 0.95464 \pp \\
    -29.261 - 0.95464 \pp & -1.6795 - 39.184 \pp
  \end{bmatrix}\!,    \\
  \Cr(\pp)
   & =
  \begin{bmatrix}
    1.1211 & -0.019113 & 1.7966 & 0.65666
  \end{bmatrix}\!.
\end{align*}
Then, we check whether the converged \ac{rom} satisfies the newly developed
optimality conditions.
The relative errors in Lagrange interpolation of $G$ in \Cref{thm:h2l2-cond-dyn}
for the two complex conjugate pairs are
$8.496 \times 10^{-9}$ and $2.105 \times 10^{-8}$.
For interpolation of $\partial G / \partial s_a$, we have
$1.6114 \times 10^{-8}$ and $1.2166 \times 10^{-7}$.
Finally, for interpolation of $\partial G / \partial s_b$, we have
$4.9016 \times 10^{-8}$ and $2.0029 \times 10^{-7}$.
Thus, the resulting optimal \ac{rom} satisfies the interpolatory optimality
conditions (to the accuracy of the stopping criteria of the minimization
algorithm).

\subsection{Parametric Penzl's FOM}

Next we consider a parametric variant of Penzl's FOM from the SLICOT benchmark
collection~\cite{slicot}.
Specifically, we consider \iac{fom} of order $16$ with
$\Ef(\pp) = I$,
\begin{gather*}
  \Af(\pp) = \mydiag*{\Af_1(\pp), \Af_2}, \\
  \Af_1(\pp) =
  \begin{bmatrix}
    -1   & \pp \\
    -\pp & -1
  \end{bmatrix}\!,\
  \Af_2 = \mydiag{-1, -2, \ldots, -10}, \\
  \Bf(\pp)\tran = \Cf(\pp) =
  \begin{bmatrix}
    5 & 5 & 1 & 1
  \end{bmatrix}\!,
\end{gather*}
and $\pset = [1, 100]$.

Setting the initial \ac{rom} based on the \ac{fom} truncated to the first $3$
states and running minimization, we obtain
$\Er(\pp) = I$,
\begin{align*}
  \Ar(\pp)
   & =
  \mydiag*{\Ar_1(\pp), \Ar_2(\pp)},              \\
  \Ar_1(\pp)
   & =
  \begin{bmatrix}
    -1.0030                & 2.2567 \times 10^{-3} \\
    -2.2567 \times 10^{-3} & -1.0030
  \end{bmatrix} \\
   & \quad
  + \pp
  \begin{bmatrix}
    7.2387 \times 10^{-6} & 1.0000                \\
    -1.0000               & 7.2387 \times 10^{-6}
  \end{bmatrix}\!,  \\
  \Ar_2(\pp)
   & =
  {-3.5530} + 2.4940 \times 10^{-4} \pp,         \\
  \Br(\pp)\tran
   & =
  \begin{bmatrix}
    2 & 0 & 1
  \end{bmatrix}\!,                               \\
  \Cr(\pp)
   & =
  \begin{bmatrix}
    25.063 & -0.053279 & 8.7695
  \end{bmatrix}\!.
\end{align*}
The relative errors in the Lagrange and the two Hermite conditions
for the first complex conjugate pair of poles are
\[
  1.0660 \times 10^{-10},\
  1.9085 \times 10^{-9}, \text{ and }
  1.5356 \times 10^{-9}.
\]
For the real pole they are
\[
  4.4460 \times 10^{-10},\
  4.4054 \times 10^{-10}, \text{ and }
  1.6685 \times 10^{-9}.
\]
Therefore, similar to the previous example,
we find good agreement with the theory.

\section{Conclusions}%
\label{sec:conclusion}

We derived interpolatory necessary optimality conditions for
$\HtwoLtwo$-optimal reduced-order modeling of
parametric \ac{lti} systems with general diagonal structure.
Then we give conditions for special cases where only inputs and outputs or
only the dynamics are parameterized.
Future work includes the derivation of an iterative, IRKA-like algorithm to
compute $\HtwoLtwo$-optimal \acp{strom}.

